\documentclass[a4paper, 11pt]{amsart}

\input xy
\xyoption{all}

\newcounter{ExacSeq}

\usepackage[normalem]{ulem} 
\usepackage{amsmath,amsthm,amssymb}
\usepackage[page,header]{appendix}
\usepackage{morefloats}
\usepackage{color, cancel}
\usepackage{fancybox}
\usepackage{tikz}
\usetikzlibrary{matrix}
\usepackage{arcs}
\usepackage{hyperref}
\usepackage{verbatim}

\usepackage[T1]{fontenc}

\makeindex 

\usepackage{latexsym, amsmath, amssymb, amsfonts, amscd}
\usepackage{amsthm}
\usepackage{t1enc}
\usepackage[mathscr]{eucal}
\usepackage{indentfirst}
\usepackage{graphicx, pb-diagram}
\usepackage{enumerate}
\usepackage[all,poly,web,knot]{xy}
\usepackage{float}
\restylefloat{figure}
\usepackage{cancel}

\newcommand{\Title}{Title}

\numberwithin{equation}{section}

{\theoremstyle{definition}\newtheorem{definition}{Definition}[section]

\newtheorem{defititle}[definition]{\Title}

\newtheorem{remark}[definition]{Remark}

\newtheorem{ex}[definition]{Example}

\newtheorem{exs}[definition]{Examples}}
\newtheorem{prop}[definition]{Proposition}
\newtheorem{proposition-definition}[definition]{Proposition-Definition}
\newtheorem{lemma}[definition]{Lemma}
\newtheorem{thm}[definition]{Theorem}
\newtheorem{cor}[definition]{Corollary}

\newtheorem*{prop*}{Proposition}
\newtheorem*{theorem*}{Theorem}

\newcounter{enumi_saved}
\newcommand{\pauseenumerate}{\setcounter{enumi_saved}{\value{enumi}}}
\newcommand{\resumeenumerate}{\setcounter{enumi}{\value{enumi_saved}}}

\newcommand{\cD}{\mathcal{D}}
\newcommand{\cG}{\mathcal{G}}

\newcommand{\cF}{\mathcal{F}}
\newcommand{\cE}{\mathcal{E}}

\newcommand{\cU}{\mathcal{U}}

\newcommand{\bL}{\mathbb{L}}

\newcommand{\id}{{\hbox{id}}}

\newcommand{\cf}{{\it cf.}\/ }

\newcommand{\hol}{\mathrm{hol}}

\newcommand{\Cp}{C_{\mathrm{p}}}
\newcommand{\Cc}{C_{\mathrm{c}}}
\newcommand{\srtimes}{\,{}_s\!\!\times_r}

\def\gpd{\,\lower1pt\hbox{$\longrightarrow$}\hskip-.24in\raise2pt
\hbox{$\longrightarrow$}\,}

\renewcommand{\latticebody}{\drop@{ }}

\newcommand{\R}{\ensuremath{\mathbb R}}

\newcommand{\cX}{\mathcal{X}}

\newcommand{\cN}{\mathcal{N}}

\newcommand{\rmt}{\mathrm{t}}

\newcommand{\Dp}{\cD_\mathrm{p}}

\newcommand{\NN}{\ensuremath{\mathbb N}}

\newcommand{\RR}{\ensuremath{\mathbb R}}

\DeclareMathOperator{\pr}{pr} 

\DeclareMathOperator{\Exp}{Exp}
\DeclareMathOperator{\supp}{supp}
\DeclareMathOperator{\Op}{Op}

\newcommand{\bfa}{\mathbf{a}}
\newcommand{\bfb}{\mathbf{b}}
\newcommand{\bfX}{\mathbf{X}}

\def\act{\mathbin{\hbox{$<\kern-.4em\mapstochar\kern.4em$}}}
\def\ract{\mathbin{\hbox{$\mapstochar\kern-.3em>$}}}

\def\PB(#1,#2,#3,#4){\left\{\begin{matrix}#1&\!\!\!\stackrel{?}{\longrightarrow}&\!\!\!#2\\
\downarrow&&\!\!\!\downarrow\\
#3&\!\!\!\stackrel{?}{\longrightarrow}&\!\!\!#4\end{matrix}\right\}}

\def\pb(#1,#2,#3,#4){ \hom(#1 \to #3, #2 \to #4)}

\begin{document}

\title[The algebra of Schwartz kernels along a singular foliation]{The convolution algebra of Schwartz kernels along a singular foliation}

\author{Iakovos Androulidakis}
\address{National and Kapodistrian University of Athens \\
 Department of Mathematics \\
 Panepistimiopolis \\
 GR-15784 Athens, Greece
}
\email{iandroul@math.uoa.gr}

\author{Omar Mohsen}
\address{Mathematisches Institut der WWU M\"unster, Einsteinstra\ss{}e 62, 48149 M\"unster, Germany.
}
\email{omohsen@uni-muenster.de}

\author{Robert Yuncken}
\address{Laboratoire de Math\'ematiques Blaise Pascal\\ 
          Université Clermont Auvergne, CNRS, LMBP, F-63000 Clermont-Ferrand, France
}
\email{robert.yuncken@uca.fr}

\subjclass[2010]{}

\thanks{O.~Mohsen and R.~Yuncken were supported by the project SINGSTAR of the Agence Nationale de la Recherche (ANR-14-CE25-0012-01). R.~Yuncken was also supported by the CNRS PICS project OpPsi.}


\begin{abstract}
  Motivated by the study of H\"ormander's sums-of-squares operators and their generalizations, we define the convolution algebra of transverse distributions associated to a singular foliation.  We prove that this algebra is represented as continuous linear operators on the spaces of smooth functions and generalized functions on the underlying manifold, and on the leaves and their holonomy covers.  This generalizes Schwartz kernel operators to singular foliations.  We also define the algebra of smoothing operators in this context and prove that it is a two-sided ideal.
\end{abstract}

\maketitle

\tableofcontents

\section{Introduction}

The goal of this article is to introduce a convolution algebra of distributions on the holonomy groupoid of a singular foliation \cite{Pradines:graph, DebordJDG, AS1}\footnote{Pradines's work \cite{Pradines:graph}, which was clarified by Debord \cite{DebordJDG}, treats only the case where the holonomy groupoid is a Lie groupoid.  The general case is treated in \cite{AS1}.} which generalizes the algebra of Lescure-Manchon-Vassout on a Lie groupoid \cite{LMV}.  The motivation is to lay the analytical foundations for a study of an extremely broad class of pseudodifferential operators, including as special cases:
\begin{itemize}
  \item the Heisenberg calculus and its generalizations \cite{BeaGre, Taylor:microlocal, ChrGelGloPol, Melin:preprint}, 
  \item singular extensions of the above calculi, where the symbol is defined on a family of nilpotent groups with varying dimension,
  \item H\"ormander's sums-of-square operators \cite{Hormander:SoS},
  \item pseudodifferential operators on a singular foliation \cite{AS2}.
\end{itemize}

Specifically, we define an algebra $\cE'_{r,s}(\cF)$ of \emph{transverse} or \emph{fibred} distributions associated to a singular foliation $\cF$ of a manifold $M$.  This algebra acts not only on $C^\infty(M)$ and $\cE'(M)$ as Schwartz kernel operators, but also on the leaves of the foliation and their holonomy covers.  In particular, it will contain the kernels of pseudodifferential operators on a singular foliation introduced in \cite{AS2,AS3}. In forthcoming work, we will explain how this algebra can be used to treat the operators described in the list above. To explain our motivation, we need to recall the relation between groupoids and pseudodifferential operators.
	
This begins, of course, with Connes' idea to use a deformation groupoid to study the $K$-theoretic properties of pseudodifferential operators and their symbols \cite{Connes79}.  Connes thus obtained a generalization of the Atiyah-Singer Index Theorem for (regular) foliations, which has led to an enormous number of generalizations.  For a sample, see \cite{ConnesSkandalis, MonPie, NisWeiXu, LauMonNis, AmmLauNis, DebLesNis, CarMon, So:boundary_groupoids, Monthubert, Nistor:singular, Ponge, VanErp:AS1, VanErp:AS2, BohSch, CarNisQia,  Nistor:Desingularuzation, Come:Fredholm}, and references therein.   

A second revolutionary idea appeared in the ground-breaking paper \cite{DebordSkandalis1}, where Debord and Skandalis observed that the classical pseudodifferential operators can in fact be characterized in terms of the canonical $\RR^\times_+$-action on Connes' tangent groupoid; see also \cite{DebSka:extensions, DebSka:exact_sequences}.  This idea was developed in \cite{vEY2}, and as with Connes' idea one sees that it is a quite general principle.  That is, the $\RR^\times_+$-action on different versions of the tangent groupoid can be used to obtain other well-known pseudodifferential calculi, such as the Heisenberg calculs \cite{BeaGre, Taylor:microlocal} and Melrose's $b$-calculus \cite{Melrose:APS}.

The main functional analytical tool underlying that work is the convolution algebra of transverse distributions on a Lie groupoid, as developed by Lescure-Manchon-Vassout \cite{LMV}, see also \cite{AS2}.  However, in the case of  sums-of-squares operators or pseudodifferential operators on singular foliations, the tangent groupoid will no longer be a Lie groupoid, and the construction of \cite{LMV} no longer applies.  

Instead, the relevant tangent groupoid will be an example of a holonomy groupoid $H(\cF)$ of a singular foliation $(M,\cF)$. The goal of this paper, therefore, is to generalize the results of \cite{LMV} to such groupoids.  The resulting convolution algebra $\cE'_{r,s}(\cF)$ coincides with that of \cite{LMV} in the case of a regular foliation, but contains the pseudodifferential kernels of \cite{AS2} in the singular case.  Of course, it will also contain many other kernels, just as the algebra of Lescure-Manchon-Vassout contains not just the classical pseudodifferential kernels, but also the kernels of other calculi, \emph{e.g.} \cite{BeaGre,	Melin:preprint}, as well as Fourier integral operators \cite{LesVas} and many others.

There are two constructions for this groupoid in the literature, due to Pradines \cite{Pradines:graph, DebordJDG} and Androulidakis-Skandalis \cite{AS1}.   In \cite{AS1, AS2, AS3}, this groupoid was used as a device to carry the analysis of pseudodifferential operators along the leaves of the foliation. Recall that the difficulties with this were the following: 
\begin{itemize} \item The $s$-fibers of $H(\cF)$ are smooth manifolds \cite{Debord2013} but have varying dimension; 
\item The topology of $H(\cF)$ is quite pathological. Specifically, there are examples for which a sequence of points in different fibres can converge to an uncountably infinite set of limit points.\end{itemize}  
In \cite{AS2} it was possible to overcome these difficulties by working on an appropriate class of smooth submersions to this groupoid, called ``bisubmersions''. This enables quite sophisticated analysis.  In particular, the authors of \cite{AS2} produce a calculus of pseudodifferential operators adapted to the geometry of a singular foliation.  

\medskip

Let us explain the philosophy for defining convolution operators in such singular situations.  To begin with,  consider the simple case of operators on a smooth closed manifold $M$. Let $a$ be a distribution on $M\times M$.  Recall that the Schwartz kernel operator $\Op(a)$ defined by
\begin{equation}
 \label{eq:Schwartz_kernel1}
 (\Op(a)f) (x) := \int_y a(x,y) f(y) \,dy
\end{equation}
gives a continuous linear operator $\Op(a):C^\infty(M) \to C^\infty(M)$ if and only if the kernel $a$ is smooth in the first variable, \emph{i.e.}, semiregular.  In groupoid language, this condition corresponds to requiring that $a$ be transverse to the $r$-fibration of the pair groupoid $M\times M$.  This point of view was introduced by Skandalis and the first author in \cite{AS2}. Lescure, Manchon and Vassout \cite{LMV} showed that the space of transverse distributions on an arbitrary Lie groupoid $G\rightrightarrows M$ forms an algebra under convolution, with a natural representation as Schwartz kernel operators on $C^\infty(M)$. These operators also act on functions on the leaves of the foliation underlying $G\rightrightarrows M$, as well as the $s$-fibers of $G$.

The apparatus developed in \cite{LMV} cannot be applied directly to a singular holonomy groupoid, because of the difficulties we described above. Instead, as in \cite{AS1,AS2,AS3}, we must lift all the analysis to bisubmersions.  We will recall the definition of a bisubmersion in Section \ref{sec:groupoid}.  For now, suffice it to say that bisubmersions act like the charts on a Lie groupoid, except that they may be only local submersions,  instead of local diffeomorphisms.
		
For instance, in order to define linear operators on $C^\infty(\RR^n)$, we have the standard form \eqref{eq:Schwartz_kernel1} above.  But nothing is stopping us from adding additional dimensions to the kernel $a$.  If $b$ is a distribution on $\RR^n \times \RR^k \times \RR^n$, smooth in $x$, then we could define an operator $\Op(b): C^\infty_c(\RR^n) \to C^\infty(\RR^n)$ by
\begin{equation}
 \label{eq:Schwartz_kernel2}
 (\Op(b)f)(x) = \int_{(z,y)\in\RR^k\times\RR^n} b(x,z,y) f(y) \,dz\,dy.
\end{equation}
(The support of $b$ needs to be compact in the $z$-direction for this to converge, but let's defer discussion of support conditions until later.)  For example, if $\varphi\in \Cc^\infty(\RR^n)$ is a smooth function of total mass $1$, then putting $b(x,z,y) = a(x,y)\varphi(z)$ we obtain exactly the same operator as in \eqref{eq:Schwartz_kernel1}.

Of course, adding these extraneous dimensions is completely unnecessary in this situation, since the Schwartz kernel theorem tells us that all continuous linear operators on $C^\infty(M)$ can be written uniquely in the form \eqref{eq:Schwartz_kernel1}.  But as pointed out in \cite{AS2}, these additional dimensions are crucial for making sense of pseudodifferential operators on singular foliations, where the dimensions of the leaves can vary from point to point.

\medskip

The paper is structured as follows.  In Section \ref{sec:groupoid} we recall the construction of the holonomy groupoid of a singular foliation $(M,\cF)$ from \cite{AS1}.  As mentioned above, its building blocks are bisubmersions and  all of our constructions throughout the paper are carried out at the level of bisubmersions rather than on the holonomy groupoid per se.
	
In Section \ref{sec:distributions} we define $r$-fibred distributions on bisubmersions and discuss various algebraic operations, including convolution and transposition. The algebra $\cE'_r(\cF)$ of properly supported $r$-fibred distributions on the holonomy groupoid is constructed in Section \ref{sec:ErF}. In \S \ref{sec:Op} we discuss the action of $\cE'_r(\cF)$ on $C^{\infty}(M)$. The right ideal of smooth $r$-fibred densities is discussed in \S \ref{sec:smooth_ideal}. In Section \ref{sec:proper} we define transverse distributions on the holonomy groupoid.  Roughly speaking, these are distributions which can be realized as smooth families of distributions on both the $r$-fibres and the $s$-fibres.  In \S \ref{sec:actgen} we show that the transverse distributions act on each of the spaces $C^{\infty}(M)$, $C^{\infty}_c(M)$, $\cE'(M)$ and $\cD'(M)$. Finally, in Section \ref{sec:actholcov}, we give the analogous actions on a leaf and on its holonomy cover.

\subsection{Acknowledgements}

The authors would like to thank Georges Skandalis for inspiration and encouragement, and Erik van Erp for many discussions.  We also thank the anonymous referees for their suggestions, which have enhanced the importance of the results.

\section{The path holonomy groupoid of a singular foliation}

\label{sec:groupoid}

Let $M$ be a smooth finite dimentional manifold. Following \cite{AS1}, we define a singular foliation $(M,\cF)$ as a $C^{\infty}(M)$-submodule $\cF$ of the module $\cX_c(M)$ of compactly supported vector fields of $M$ which is locally finitely generated and involutive.

\subsection{Bisubmersions}\label{sec:bisubm}

To keep this paper self-contained, we recall here the notion of a bisubmersion from \cite{AS1}. For the reader unfamiliar with bisubmersions, keep in mind that the archetypal example is an open subset $U$ of a Lie groupoid $G$ over $M$, equipped with the restrictions of the range and source maps $r$ and $s$.  In that example, the underlying foliation $\cF$ is the foliation of $M$ by the orbits of $G$. Specifically, the module $\cF$ is the image by the anchor map $\rho : AG \to TM$ of the $C^{\infty}(M)$-module of compactly supported sections $\Gamma_c(AG)$. 

General bisubmersions replace charts for more singular groupoids.  They serve as a lifting of a region in the singular space to a nice locally euclidean manifold.

Let $(M,\cF)$ be a foliation.  If $\varphi : N \to M$ is a smooth map between manifolds, we write
\[
\varphi^{-1}\cF = \{ Y \in \cX_c(N) : d\varphi\circ Y = \sum_{i=1}^n f_i(X_i \circ \varphi) \text{ for } f_i \in C^{\infty}_c(N), X_i \in \cF \}.
\]
If $\varphi$ is a submersion, which will always be the case in what follows, this means that the elements of $\varphi^{-1}\cF$ are vector fields on $N$ which project under $\varphi$ to vector fields in $\cF$.  Then $\varphi^{-1}\cF$ is a also a singular foliation (locally finitely generated and involutive).

\begin{enumerate}

\item 
A \emph{bisubmersion} $(U,r_U,s_U)$ of $(M,\cF)$ is a smooth, finite dimensional, Hausdorff manifold $U$ equipped with two submersions $r_U,s_U:U \to M$, called \emph{range} and \emph{source}, such that
\[
s_U^{-1}\cF = r_U^{-1}\cF = C_c^\infty(U;\ker ds_U) + C_c^\infty(U;\ker dr_U).
\]
Recall \cite{AS1} that this condition implies the following geometric property: If $L_x$ is the leaf of $(M,\cF)$ at $x \in M$ then $r$ maps the fiber $s^{-1}(x)$ to $L_x$ and it is a submersion. Likewise, $s : r^{-1}(L_x) \to L_x$ is a submersion.

We will often blur the distinction between a bisubmersion $(U,r_U,s_U)$ and its underlying space $U$, and we write $r$ and $s$ instead of $r_U$ and $s_U$ when the bisubmersion $U$ is evident from the context.

\item
A \emph{morphism} of bisubmersions from $(U,r_U,s_U)$ to $(V,r_V,s_V)$ is a smooth map $\varphi:U \to V$ such that $r_U = r_V\circ\varphi$ and $s_U = s_V\circ\varphi$.  A \emph{local morphism} at $u\in U$ is a morphism of bisubmersions from from $(U',r_U,s_U)$ to $(V,r_V,s_V)$ for some neighbourhood $U'$ of $u$.

\item
The \emph{inverse} of a bisubmersion $(U,r_U,s_U)$ is the bisubmersion $(U,s_U,r_U)$.  We will sometimes use $U^\rmt$ to denote the space $U$ equipped with this inverse bisubmersion structure.

\item
The \emph{composition} $(U\circ V , r_{U\circ V}, s_{U\circ V})$ of two bisubmersions $(U,r_U,s_U)$ and $(V,r_V,s_V)$ is the bisubmersion with $U\circ  V := U {}_{s_U}\times_{r_V} V$ and maps $r_{U\circ V}(u,v) = r_U(u)$ and $s_{U\circ V}(u,v) = s_V(v)$.

More generally, if $A \subset U$ and $B\subset V$, we write $A\circ B = A{}\srtimes B$.

\item
A bisection of a bisubmersion $U$ is a locally closed submanifold $S \subset U$ such that the restrictions of $r_U$ and $s_U$ to $S$ are diffeomorphisms onto open subsets of $M$.  Any bisection $S$ induces a local diffeomorphism $\Phi_S$ on $M$ 
\[
 \Phi_S = r_U|_S \circ s_U|_S^{-1}.
\]
If $\Phi_S$ is the identity on its domain, $S$ is called an \emph{identity bisection}.

\pauseenumerate
\end{enumerate}

We will really only be interested in bisubmersions of the following particular type. Let $\bfX=(X_1,\ldots, X_m)$ be a generating family of vector fields for $\cF$ in an open subset $M_0 \subseteq M$.  Consider the map
\begin{equation}
 \label{eq:Exp}
\Exp_{\bfX} : \R^m \times M_0 \to M
\end{equation}
where $\Exp_{\bfX}(\xi,x)$ is the unit-time flow of $x$ along the vector field $\sum_i \xi_i X_i$, when defined. This map is defined on some sufficiently small neighbourhood $U$ of $\{0\} \times M_0$ in $\RR^m\times M_0$.

\begin{definition}[\cite{AS1}]
 \label{def:path_holonomy_bisubmersion}
 With the above notation, the set $U \subseteq \RR^n\times M_0$ is a bisubmersion when equipped with the maps
 \[
  s_U(\xi,x) = x, \qquad r_U(\xi,x) = \Exp_{\bfX}(\xi,x).
 \]
 We call $(U,r_U,s_U)$ a \emph{path-holonomy bisubmersion} associated to the local generating family $\bfX=(X_1,\ldots, X_m)$.
 If $\bfX$ is a minimal generating family at $x\in M_0$ then $(U , r_U, s_U)$ is called a \emph{miminal path-holonomy bisubmersion at $x$}.
\end{definition}

\subsection{Atlas of bisubmersions}\label{sec:atlas}

The holonomy groupoid of a singular foliation is defined in terms of an atlas of bisubmersions, as follows.

\begin{enumerate}
\resumeenumerate

\item 
A bisubmersion $(V,r_V,s_V)$ is \emph{adapted} to a family $\cU$ of bisubmersions if for every $v\in V$ there exists a local morphism at $v$ from $(V,r_V,s_V)$ to some bisubmersion in $\cU$.

\item
A family $\cU$ of second countable bisubmersions of $(M,\cF)$ is called a \emph{singular groupoid atlas}, or just \emph{atlas}, if the source images $\{s(U) :U\in\cU\}$ cover $M$ and all inverses and compositions of elements of $\cU$ are adapted to $\cU$.

\item
An atlas $\cU$ is called \emph{maximal} if every bisubmersion which is adapted to $\cU$ is already in $\cU$. Any atlas $\cU$ can be completed to the maximal atlas $\tilde{\cU}_\mathrm{max}$ of all bisubmersions adapted to $\cU$.

\item
Given any family $\cU_0$ of bisubmersions of $(M,\cF)$ whose source images $\{s_U(U) : U\in\cU_0\}$ cover $M$, the \emph{minimal atlas generated by $\cU_0$} is the set of all iterated compositions of elements of $\cU_0$ and their inverses, and the \emph{maximal atlas generated by $\cU_0$} is the maximal completion of this atlas in the sense above.

\end{enumerate}

\begin{definition}
 \label{def:Uhol}
 Let $(M,\cF)$ be a singular foliation.  We write $\cU_\hol(\cF)$, or just $\cU_\hol$, for the maximal atlas generated by all path-holonomy bisubmersions.
\end{definition}

\begin{remark}
\label{rmk:set-theory}
 Let us make a couple of technical remarks about this definition.  Firstly, a maximal atlas is too large to be a set.  This doesn't actually matter in what follows, but if desired, the problem could averted  by allowing only bisubmersions where $U$ is an embedded submanifold of $\RR^n$ for some $n$.

 Secondly, in \cite{AS2}, the authors work mainly with the minimal path-holonomy atlas, and not its maximal completion.  This doesn't change anything in practice.  We are favouring the maximal atlas because it slightly simplifies the natural equivalence relation for distributional kernels, see Section \ref{sec:ErF}. 
 
 In this article for simplicity we will only work with maximal atlases.
\end{remark}

\begin{exs}\label{exs:regatlas}
\begin{enumerate}
\item
Suppose that the foliation $(M,\cF)$ is defined by a Lie groupoid $\cG \gpd M$. Then we can consider the maximal atlas generated by a single bisubmersion, which is none other than $\cG$. This atlas contains all finite compositions $\cG \circ\ldots\circ \cG$.  The multiplication $\cG \circ\ldots\circ \cG \to \cG$ and inversion $\cG^\rmt \to \cG$ maps are morphisms.

\item 
Recall from \cite[Ex. 3.4 \& \S 3.2]{AS1} that when the foliation $(M,\cF)$ is regular, its holonomy groupoid $H(\cF)$ constructed by Winkelnkemper is a Lie groupoid, so it admits an atlas generated by a single bisubmersion. It was also shown in \cite{AS1} that this atlas is equivalent with the path-holonomy atlas $\cU_\hol$.
\end{enumerate}
\end{exs}

\subsection{Holonomy groupoid}
\label{sec:holonomy_groupoid}

An atlas of bisubmersions has an associated singular groupoid (\cf \cite{AS1}).  We won't ever actually use this groupoid in what follows, since we'll define the convolution algebra of fibred distributions directly on the atlas of bisubmersions.  But for the sake of completeness, let us finish this section with the construction.

\begin{definition}[\cite{AS1}]
 \label{def:groupoid}
 Let $\cU$ be an atlas of bisubmersions for the foliation $(M,\cF)$. The \emph{groupoid associated with $\cU$} is $\cG(\cU) = \coprod_{U \in \cU}U/\sim$ where the equivalence relation is defined as follows: $U \ni u \sim v \in V$ if there exists a local morphism from $(U,r_U,s_U)$ to $(V,r_V,s_V)$ sending $u$ to $v$.  
\end{definition}

This is indeed a topological groupoid with base $M$, see \cite[Proposition 3.2]{AS1}. It is proved in \cite{Debord2013} that it is longitudinally smooth.  More specifically, the $s$-fibres $H(\cF)_x$ are smooth manifolds and $r$ restricts to a surjective submersion from $H(\cF)_x$ to the leaf $L_x$ through $x$.   Moreover, if we put $H(\cF)_{L_x} = s^{-1}(L_x) = r^{-1}(L_x)$, then $r : H(\cF)_{L_x} \to L_x$ is also a surjective submersion.

In \cite[\S 3.2]{AS1} it is shown that, when the module $\cF$ is projective, then the groupoid $\cG(\cU_{hol})$ is a Lie groupoid and coincides with the holonomy groupoid constructed by Debord. In particular, when the foliation is regular, $\cG(\cU_{hol})$ coincides with the holonomy groupoid $H(\cF)$.

\section{Distributions on a singular groupoid}
\label{sec:distributions}

\subsection{Fibred distributions on a submersion}\label{sec:fibdistrsubm}

We recall some basic definitions concerning fibred distributions.  This material is adapted from \cite{LMV}.  The basic ideas appear already in \cite{AS2}.

Let $q:U\to M$ be a submersion.  Then $C^{\infty}(U)$ becomes a left $C^{\infty}(M)$-module with 
\[
 f \cdot \phi = (q^{*}f)\phi
\]
for all $\phi \in C^{\infty}(U)$ and $f \in C^{\infty}(M)$.

\begin{enumerate}
\item A \emph{properly supported $q$-fibred distribution} is a continuous linear map 
\[
a : C^\infty(U) \to C^\infty(M); \qquad \phi \mapsto (a,\phi)
\]
which is $C^\infty(M)$-linear with respect to the above $C^\infty(M)$-structure on $C^\infty(U)$. The continuity of $a$ is understood with respect to the usual topology on $C^{\infty}(M)$ and $C^{\infty}(U)$ of  uniform convergence of derivatives on compact subsets.

\item
The space of properly supported $q$-fibred distributions on $U$ is denoted $\cE'_q(U)$.
We equip $\cE'_q(U)$ with the topology inherited as a closed subspace $\bL(C^\infty(U), C^\infty(M))$ with the topology of uniform convergence on bounded subsets.

\item\label{part c} Let $x\in M$.  By $C^\infty(M)$-linearity, the value of $(a,\phi)$ at $x$ depends only on the restriction of $\phi$ on the fibre $q^{-1}(x)$, and this defines a compactly supported distribution $a_x \in \cE'(q^{-1}(x))$.  The $q$-fibred distribution $a$ is uniquely determined by the family of distributions $(a_x)_{x\in M}$.

\item
Let $V\subseteq U$ be open.  The $q$-fibred distribution $a$ \emph{vanishes on $V$} if $(a,\phi) = 0$ whenever $\supp(\phi) \subseteq V$.  The \emph{support} of $a$, denoted by $\supp(a)$, is the complement of the largest open subset of $U$ on which $a$ vanishes.  For any $a\in\cE'_q(U)$, the support $\supp(a)$ is a $q$-proper set, meaning that the restriction $q : \supp(a) \to M$ is a proper map.
\end{enumerate}

\begin{ex}
 \label{ex:Dirac_distribution}
 Let $S \subset U$ be a local section of $q:U\to M$, meaning that $S$ is a locally closed submanifold of $U$ and $q|_S$ is a diffeomorphism of $S$ onto an open subset of $M$.  
 Fix also a smooth function $c\in \Cc^\infty(M)$ with support in $q(S)$.   The \emph{$q$-fibred Dirac distribution on $S$ (with coefficient $c$)}, denoted $c\Delta_S\in\cE'_q(U)$, is the $q$-fibred distribution defined by the formula
 \[
 (c\Delta_S, \phi)(x) =
 \begin{cases}
 c(x)(\phi \circ q|_S^{-1})(x),  
 & \text{if } x\in q(S) \\
 0,
 & \text{otherwise,}
 \end{cases}
 \]
for $\phi\in C^\infty(U)$.
In other words, upon identifying $S$ with its image $q(S)\subset M$, $c\Delta_S$ is given by evaluating $\phi$ on $S$ and then multiplying by $c$ (to ensure the result is smooth).  For any $x\in q(S)$, the restriction of $c\Delta_S$ to the fibre $q^{-1}(x)$ is the multiple $c(x)\delta_{\tilde{x}}$ of the Dirac distribution at the preimage $\tilde{x}=q|_S^{-1}(x)$.
\end{ex}

\begin{ex}
	\label{ex:PsiDOs}
	Let $S\subset U$ be an identity bisection of a bisubmersion $U$.  If $a$ is a smooth family of distributions on the $r$-fibres of $U$ with pseudodifferential singularities along $S$, as in\cite[\S1.2.2]{AS2},  then $a$ defines an element of $\cE'_r(U)$. 
\end{ex}

We end this section with two operations on fibred distributions which will be heavily used in what follows.

\subsubsection{Pushforward, or integration along fibres}

\begin{definition}\label{dfn:pushforward}
Let $q:U \to M$ and $q':U'\to M$ be submersions and let $\pi:U'\to U$ be a smooth morphism of submersions, 
meaning that $q' = q\circ\pi$. There is an induced linear map $\pi_*:\cE'_{q'}(U') \to \cE'_{q}(U)$, called \emph{pushforward} or \emph{integration along the fibres}, defined by
\[
(\pi_*a,\phi) := (a,\pi^*\phi)
\]
for $a\in\cE'_{q\circ\pi}(U')$ and $\phi\in C^\infty(U)$.
\end{definition}

\begin{ex}
	\label{ex:extension_by_zero}
 Let $U'$ be an open subset of a bisubmersion $U$.  The inclusion map $\iota:U' \to U$ is a morphism of bisubmersions.  The pushforward $\iota_*:\cE'_q(U') \to \cE'_q(U)$ corresponds to extension by zero.  We can thus identify $\cE'_q(U')$ as a subspace of $\cE'_q(U)$.
\end{ex}

\begin{lemma}
 \label{lem:integration_is_surjective}
 If the morphism of submersions $\pi: U' \to U$ is a surjective submersion, then $\pi_*:\cE'_{q'}(U') \to \cE'_{q}(U)$ is surjective.
\end{lemma}

\begin{proof}
 We begin with the case where $U'=U \times \R^k$ for some $k\in\NN$ and $\pi$ is the projection onto the first variable.  Fix a positive function $\omega\in C_c^\infty(\RR^n)$ with $\int_{\RR^n} \omega(\xi)\, d\xi = 1$. If $a\in \cE'_q(U)$, then we define an $a'\in \cE'_{q'}(U')$ by the pairing
 \[
  (a', \phi) = (a, \textstyle\int_{\R^k} \phi(\;\cdot\;,\xi) \omega(\xi) \,d\xi ) \qquad \phi \in C^\infty(U\times \R^k).
 \]
 Then $\pi_*a'=a$, as desired.

The general case follows by from the above by using a partition of unity argument.
\end{proof}

\subsubsection{Pullback over a base map}

The following construction is described in \cite[Proposition 2.15]{LMV}.

\begin{definition} 
 \label{def:pullback}
 Let $p:N\to M$ be any smooth map, and consider the pullback diagram
 \[
  \xymatrix{
    N {}_p\!\times_q U \ar[r]^-{\pr_U} \ar[d]_{\pr_N} 
    & U \ar[d]^{q} \\
    N \ar[r]^{p} 
    & M.
    }
 \]
 There is an associated linear map $p^* : \cE'_q(U) \to \cE'_{\pr_N}(N {}_p\!\times_q U)$, called the \emph{pullback along $p$}, which is uniquely characterized by the property
 \begin{equation}
 \label{eq:pullback_distribution}
  (p^*a, \pr_U^*\phi) = p^*(a,\phi)
 \end{equation}
 for all $a\in\cE'_q(U)$, $\phi\in C^\infty(U)$. Explicitly, $p^*a$ is defined on the fibres of $\pr_N$ by $(p^*a)_y=a_{p(y)}$, where $y\in N$ and we are using the canonical identification $\pr_N^{-1}(y) = q^{-1}(p(y))$.
\end{definition}

\begin{lemma}
 \label{lem:pullback-functoriality}
 Let $q:U\to M$ be a submersion and $p:N\to M$ a smooth map, as above.

 \begin{enumerate}
  \item If $q':U'\to M$ is a submersion and $\pi:U'\to U$ a morphism of submersions, then for any $a\in \cE'_{q'}(U')$,
  \[
   p^*(\pi_*a) = (\id\times\pi)_*(p^*a).
  \]

  \item Let $p':N'\to N$ be a smooth map, then $(p\circ p')^*= p'^*\circ p^*$. 
   \end{enumerate}
\end{lemma}

\begin{proof}

\begin{enumerate}
  \item  Let $\phi \in C^\infty(U)$. One has 
  \begin{align*}
   ((\id\times\pi)_*p^*a,\pr_U^*\phi) 
    &= (p^*a, (\pr_U\circ(\id\times\pi)))^*\phi) \\
    &= (p^*a,(\pi\circ\pr_{U'})^*\phi) \\
    &= p^*(a,\pi^*\phi) \\
    &= p^*(\pi_*a,\phi).  \end{align*}
The result then follows from uniqueness in Equation \eqref{eq:pullback_distribution}.

  \item Let $\pr'_{U}:N' {}_{p\circ p'}\!\!\times_q U\to U$ be the natural projection, $\phi \in C^\infty(U)$. One has
\begin{align*}
({p'}^*(p^*a),(\pr'_{U})^*\phi)=p'^*(p^*a,\pr_U^*\phi)={p'}^*p^*(a,\phi)=(p\circ p')^*(a,\phi).
\end{align*} The result then follows from uniqueness in Equation \eqref{eq:pullback_distribution}.%
\qedhere
 \end{enumerate}

\end{proof}

Finally, we remark that pushforward and pullback are continuous with respect to the topologies on fibred distributions.

\subsection{Convolution of fibred distributions on bisubmersions}

Now we consider fibred distributions on a bisubmersion $U$ for a foliation $(M,\cF)$.  Thus we have two submersions $r,s:U \to M$ and we can define the spaces $\cE'_r(U)$ and $\cE'_s(U)$ of linear maps from $C^\infty(U)$ to $C^\infty(M)$.  For $a\in \cE'_s(U)$, we will write $a_x$ for the distribution on the $s$-fibre $s^{-1}(x)$ and for $b\in \cE'_r(U)$ we write $b^x$ for the distribution on the $r$-fibre $r^{-1}(x)$.

Recall that if $U$ and $V$ are bisubmersions over $M$, then their composition is $U \circ V = U {}_s\!\times_r V$ with range and source maps $r_{U \circ V}(u,v) = r(u)$ and $s_{U \circ V}(u,v) = s(v)$. 
Therefore, given $b\in\cE'_r(V)$ we can define a pullback distribution $s_U^*b \in \cE'_{\pr_U}(U\circ V)$ as in Definition \ref{def:pullback} above, and hence make the following definition.

\begin{definition} \label{def:convolution}
 Let $U$, $V$ be bisubmersions for $(M,\cF)$.  We define the \emph{convolution product} of $a\in \cE'_r(U)$ and $b\in\cE'_r(V)$ to be the $r$-fibred distribution
 \[
   a*b := a \circ s_U^*b \in \cE'_r(U\circ V).
 \]
 Similarly, the \emph{convolution product} of $a\in \cE'_s(U)$ and $b\in\cE'_s(V)$ is the $s$-fibred distribution
 \[
  a * b := b \circ r_V^*a \in \cE'_s(U\circ V).
 \]

 The \emph{transpose} $a^\rmt\in\cE'_s(U^\rmt)$ of an $r_U$-fibred distribution $a\in\cE'_r(U)$ is $a^\rmt := a$ but viewed as an $s_{U^\rmt}$-fibred distribution on the inverse bisubmersion $U^\rmt = U$.  Likewise, the transpose of an $s$-fibred distribution $b\in\cE'_s(U)$ is $b^\rmt = b \in \cE'_r(U^\rmt)$.
\end{definition}

Note that convolution  $\cE'_r(U) \times \cE'_r(V) \to \cE'_r(U\circ V)$ is separately continuous, since it is built from the continuous operations of pullback and composition.  Likewise for convolution of $s$-fibred distributions and transposition.
\begin{remark}
In the definition of convolution, we should not exclude the possibility of the empty bisubmersion $(\emptyset, r_\emptyset, s_\emptyset)$ where $r_\emptyset$ and $s_\emptyset$ are the empty maps.  Here, convention says that $C^\infty(\emptyset) = \{0\}$, so that $\cE'_r(\emptyset)$ and $\cE'_s(\emptyset)$ contain only the zero map.  This comes into play when considering a convolution of $r$-fibred distributions $a\in\cE'_r(U)$ and $b\in\cE'_s(V)$ such that $s(U)$ and $r(V)$ are disjoint, since then $U\circ V = \emptyset$ and hence $a*b=0$.
\end{remark}

This is a particular case of the following lemma.

\begin{lemma}
 \label{lem:disjoint_support}
 For any $a\in \cE'_r(U)$ and $b\in\cE'_r(V)$ we have 
 \[
  \supp(a\ast b) \subseteq \supp(a)\circ\supp(b).
 \]
 In particular, if $s(\supp(a)) \cap r(\supp(b)) = \emptyset$ then $a*b=0$.

 Analogous statements hold for $a\in \cE'_s(U)$, $b\in\cE'_s(V)$.
\end{lemma}

The convolution product is associative. It is also compatible with integration along fibres in the following sense.

\begin{lemma}
 \label{lem:convolution_integration_compatible}
 If $\pi:U\to U'$ and $\rho:V \to V'$ are submersive morphisms of bisubmersions then $\pi\times\rho:U\circ V \to U'\circ V'$ is a submersive morphism of bisubmersions and we have $(\pi_*a)*(\rho_*b) = (\pi\times\rho)_*(a*b)$ for all $a\in\cE'_r(U)$, $b\in\cE'_r(V)$ .
\end{lemma}

\begin{proof}

Let $\phi\in C^\infty(U\srtimes V)$.  Using the functorial properties of Lemma \ref{lem:pullback-functoriality} and Equation \eqref{eq:pullback_distribution}, we have
 \begin{align*}
  ((\pi_*a) * (\rho_*b), \phi)
   &= (\pi_*a,(s_{U'}^*(\rho_*b),\phi)) \\
   &= (\pi_*a, ((\id\times\rho)_*(s_U^*b), \phi)) \\
   &= (a, \pi^* (s_U^*b, (\id\times\rho)^*\phi)) \\
   &= (a, (s_U^*b, (\pi\times\id)^*(\id\times\rho)^*\phi)) \\
   &= ((\pi\times\rho)_*(a*b), \phi)\qedhere
 \end{align*}
\end{proof}

\begin{lemma}
 \label{lem:convtransp}
 For any $a\in\cE'_r(U)$ and $b\in\cE'_r(V)$ we have  
 \[
   (a*b)^\rmt = b^\rmt * a^\rmt
 \]
 as elements of $\cE'_s(U\circ V)$.
\end{lemma}

\begin{proof}
 We calculate $(a*b)^\rmt = (a\circ s_U^*b)^\rmt = a^\rmt \circ r_U^*(b^\rmt) = b^\rmt * a^\rmt$.
\end{proof}

\section{The convolution algebra of fibred distributions on the holonomy groupoid}
\label{sec:ErF}

In this section, we will define the convolution algebra $\cE_r'(\cF)$ of properly supported $r$-fibred distributions on the holonomy groupoid of a singular foliation.  In \S \ref{sec:Op}, we will show that $\cE'_r(\cF)$ acts by continuous linear operators on the spaces $C^\infty(M)$, via a representation which we call $\Op$.  These are what we refer to, informally, as the `Schwartz kernel operators' associated to a singular foliation.

There is likewise a convolution algebra $\cE_s'(\cF)$ of properly supported $s$-fibred distributions.  This algebra admits a representation as operators on the distribution space $\cE'(M)$.  For an algebra which acts at once on all four spaces $C^\infty(M)$, $\Cc^\infty(M)$, $\cD'(M)$, $\cE'(M)$, we will need to add further conditions, which we deal with in \S \ref{sec:proper}.

\subsection{The convolution algebra of $r$-fibred distributions}
Let $(M,\cF)$ be a foliation and $\cU$ a maximal atlas of bisubmersions. 

\begin{definition}
The space $\cE'_r(\cU)$ denotes the vector space of all families $(a_U)_{U\in \cU}$ that are \emph{$r$-locally finite}. This means that for every compact $K \subseteq M$ there are only finitely many $U\in \cU$ with 
 \[
  r(\supp(a_U)) \cap K \neq \emptyset.
 \]
\end{definition}

Once again, strictly speaking, this definition has some set-theoretic problems, which can be resolved as in Remark \ref{rmk:set-theory}.

If $U_0\in \cU$, then we will naturally identify $\cE'_r(U_0)$ with the subspace of $\cE'_r(\cU)$ consisting of $(a_U)_{U\in\cU}$ such that $a_U=0$ if $U\neq U_0$.
 
We will also represent an element $\bfa=(a_U)_{U\in\cU}\in \cE'_r(\cU)$ as a sum $\bfa=\sum_{U}a_U$.

\begin{prop} The convolution defined by $\sum a_U \circ \sum b_U=\sum_{U,V} a_U\ast b_V$ is well defined. 
\end{prop}

\begin{proof} 
	By this we mean that $\sum_{U,V} a_U\ast b_V$ is $r$-locally finite. Let $K\subseteq M$ be compact, and denote by $U_j$ the bisubmersions such that $r(\supp(a_{U_j}))\cap K\neq \emptyset$. Since $b$ is $r$-locally finite and $\supp(a_{U_j})\cap r^{-1}(K)$ is compact, there exists a finite number of $V_i$ such that there exists $j$ with $r(\supp{a_{U_j}})\cap s(\supp(b_{V_i})\cap r^{-1}(K))\neq \emptyset$.  The set of bisubmersions $U_j\circ V_i$ includes all those $U\circ V$ for which $r(\supp(a_U\ast b_V))\cap K \neq\emptyset$.  The result follows.
\end{proof}

We equip $\cE'_r(\cU)$ with the following topology. A generalized sequence $\bfa_i = (a_{i,U})$ converges to $\bfa = (a_U)$ if the families $\bfa_i$ are uniformly $r$-locally finite---meaning that for every compact $K\subseteq M$ there are only finitely many $U \in \cU$ for which $r(\supp(a_{i,U})) \cap K \neq \emptyset$ for some $i\in I$---and $a_{i,U} \to a_U$ for every $U$.

Now let us focus on the path-holonomy atlas $\cU_{hol}$.

\begin{definition}
 \label{def:ErF}
 We define $\cN_r \subseteq \cE'_r(\cU_{\hol})$ to be the closure of the ideal generated by all elements of the form $a - \pi_*a$, where $a\in\cE'_r(U)$ for some bisubmersion $U\in\cU_\hol$ and $\pi:U \to V$ is a morphism of bisubmersions. For elements $a\in \cE'_r(U)$, $b\in \cE'_r(V)$ we will write $a \equiv b$ when $a-b \in \cN_r$. We define
\[
\cE'_r(\cF) = \cE'_r(\cU_\hol) / \cN_r.
\]
We define spaces $\cN_s$ and $\cE'_s(\cF) = \cE'_s(\cU_\hol) / \cN_s$ analogously.
\end{definition}

\begin{remark}
	It is necessary to take a closure in the definition of the ideal $\cN_r$ in order to allow the equivalences in $\cE'_r(\cU_{hol})$ to take place on an infinite (but $r$-locally finite) family of bisubmersions.
\end{remark}

The point of quotienting by the ideal $\cN_r$ is that, as we will see in Section \ref{sec:Op}, the $r$-fibred distributions $a$ and $\pi_*a$ will induce the same kernel operators on $C^\infty(M)$.  For a simple example, if $\iota:U' \to U$ is the inclusion of an open set of a bisubmersion, then every kernel $a\in \cE'_r(U')$ is identified in the quotient with its extension by zero $\iota_*a\in \cE'_r(U)$.

\begin{remark}
	\label{rmk:PsiDOs}
The equivalence relation generated by declaring $a$ equivalent to its pushforward $\pi_*a$, is essentially the same as the equivalence relation on pseudodifferential kernels in \cite{AS2}.  The differences are purely aesthetic (use of half densities, and distributions in place of fibred distributions on bisubmersions).  In view of Example \ref{ex:PsiDOs} it follows that, up to aesthetic modifications, the kernels of pseudodifferential type defined in \cite{AS2} also define elements of the algebras $\cE'_r(\cF)$ and $\cE'_r(\cF)$.
\end{remark}

\begin{prop}
 \label{prop:convolution_algebra}
\begin{enumerate}
\item The convolution product on $\cE'_r(\cU_{hol})$ descends to a seperately continuous associative product on the quotient space $\cE'_r(\cF)$.  Similarly for $\cE'_s(\cF)$.  
\item 
\begin{sloppypar}
Transposition descends to a bijective anti-algebra isomorphism $\cE'_r(\cF) \to \cE'_s(\cF)$.
\end{sloppypar}
\end{enumerate}
\end{prop}

\begin{proof}
\begin{enumerate}
 \item 
 It follows from Lemma \ref{lem:convolution_integration_compatible} that $\cN_r$ is a closed two-sided ideal in $\cE'_r(\cU_{hol})$, which proves the first statement.  The statement for $\cE'_s(\cF)$ is proven similarly.  

 \item
 It is clear that transposition defines a continuous linear isomorphism from $\cE'_r(\cU_{hol})$ to   $\cE'_s(\cU_{hol})$ and that $\cN_r$ maps to $\cN_s$.  The result then follows from Lemma \ref{lem:convtransp}.  
\end{enumerate}
\end{proof}

\begin{exs}
\begin{enumerate}
\item 
Let us discuss the case when the foliation $(M,\cF)$ is defined by a Lie groupoid $\cG \gpd M$ as in Example \ref{exs:regatlas}. By abuse of notation, we denote $\cG$ the maximal atlas generated by $\cG$. It is then obvious that the algebra $\cE'_r(\cG)$ we construct here coincides with one constructed in \cite{LMV}.

\item 
When $\cU$ and $\cU'$ are equivalent atlases, it is routine to check that the algebras $\cE'_r(\cU)$ and $\cE'_r(\cU')$ are isomorphic. As was mentioned in Examples \ref{exs:regatlas}, when the foliation $(M,\cF)$ is almost regular (or just regular), the maximal atlas $H(\cF)$, which is equivalent to the path-holonomy atlas. Therefore, in this case our algebra $\cE_r(\cF)$ is again isomorphic to the convolution algebra constructed in \cite{LMV}.
\end{enumerate}
\end{exs}

\section{Action on smooth functions}
\label{sec:Op}

The most important feature of the convolution algebra $\cE'_r(\cF)$ is that it acts by continuous linear operators on $C^\infty(M)$, as well as $C^{\infty}$ functions on the leaves of $(M,\cF)$ and their holonomy covers. We start by explaining the action on $C^\infty(M)$.  We will treat the action on leaves and their holonomy covers in Section \ref{sec:actleaf}.

From this point on, all bisubmersions will be considered in the maximal path holonomy atlas $\cU_{hol}$.

\begin{prop}
	\label{prop:Op}
	Let $U\rightrightarrows M$ be a bisubmersion and $a\in \cE'_r(U)$. 
	\begin{enumerate}
		\item The formula
		\[
		\Op(a)f = (a, s_U^*f)
		\]
		defines a continuous linear operator $\Op(a)$ on $C^\infty(M)$.  
		
		\item If $b\in\cE'_r(V)$ for another bisubmersion $V\rightrightarrows M$, we have
		\[
		\Op(a)\Op(b) = \Op(a*b).
		\]
		
		\item If $\pi:U \to U'$ is a morphism of bisubmersions then $\Op(a) = \Op(\pi_*(a))$.
		
	\end{enumerate}
\end{prop}

\begin{proof}
 The linear maps $s_U^*:C^\infty(M) \to C^\infty(U)$ and $a:C^\infty(U) \to C^\infty(M)$ are continuous, which proves a).  
 
 The statement in b) follows from the calculation
 \begin{align*}
 \Op(a*b)f &= (a \circ s_U^*b, s_{U\circ V}^*f) \\
 &= (a, (s_U^*b, \pr_V^*s_V^*f)) \\
 &= (a, s_U^*(b,s_V^*f))  && \text{ (by Eq.\ \eqref{eq:pullback_distribution})} \\
 &= \Op(a)\Op(b)f.
 \end{align*}
 Finally, by the definition of a morphism of bisubmersions we have
 \[
 \Op(\pi_*(a))(f) 
 = (a, \pi^*s_{U'}^*f) = (a, s_U^*f) = \Op(a)(f),
 \]
 which proves c).
\end{proof}

If $\bfa  = \sum_U a_U \in \cE'_r(\cU_{hol})$, then for any $f\in C^\infty(M)$ we define
\[
\Op(\bfa) f = \sum_{U\in\cU_\hol} (a_U, s_U^*f).
\]
The sum is well-defined by $r$-local finiteness.
As an immediate consequence of Proposition \ref{prop:Op}, we have the main theorem of this section.

\begin{thm}
  The map $\Op : \cE'_r(\cU_{hol}) \to \bL(C^\infty(M))$ descends to a continuous representation of $\cE'_r(\cF)$ on $C^\infty(M)$.  \qed
\end{thm}

\begin{ex}
 \label{ex:translation_Op}
 Let $S$ be a local bisection of a bisubmersion $U$, as in item $(e)$ of \S \ref{sec:bisubm}, and let $\Phi_S = r|_S \circ s|_S^{-1}$ be the local diffeomorphism that it carries. Fix also a smooth function $c\in C^\infty(M)$ with support contained in $r(S)$.  Recall from Example \ref{ex:Dirac_distribution} that the $r$-fibred Dirac distribution supported on $S$ with coefficient $c$ is defined such that its pairing with $\phi\in C^\infty(M)$ is given by restriction to $S \cong r(S)$ and then multiplication by $c$.  We will use the notation $c\Delta^S$ for this, with $S$ in superscript to indicate it is an $r$-fibred Dirac distribution.  

We claim that for all $x\in r(S)$ we have
\[
 (\Op(c\Delta^S)f)(x) = c(x) f(\Phi_S^{-1}(x)),
\]
and $(\Op(c\Delta^S)f)(x) = 0$ when $x\notin r(S)$.
Indeed, if $x\in r(S)$ then
\[
(\Op(c\Delta^S)f)(x) = (c\Delta^S, s^* f)(x) = c(x)  f(s(r|_S^{-1}(x)).
\]
In particular, if $X\subseteq s(S)$ is a closed subset and if $c \in C^\infty(M)$ is a smooth bump function with $c|_{\Phi_S(X)} \equiv 1$ and $c|_{M\setminus r(S)} \equiv 0$, then for all $f$ supported in $X$ we have
 \[
  (\Op(c\Delta^S)f) = f\circ\Phi_S^{-1}.
 \]
 For instance, using a path-holonomy bisubmersion, we obtain flows along vector fields tangent to $\cF$ in this way.

 In particular, this example shows that the quotient algebra $\cE'_r(\cF)$ is not trivial.
\end{ex}

\subsection{Propagation of supports}

To describe propagation of supports under the action of $\Op(a)$, we introduce some further notation.

\begin{definition}
	\label{def:propagation}
	Let $U\rightrightarrows M$ be a bisubmersion.  For subsets $N \subseteq M$ and $V\subseteq U$ we define the following subsets of $M$:
	\begin{align*}
	V\circ N &= \{r(v) : v\in V \text{ with } s(v)\in N \}, \\
	N\circ V &= \{s(v) : v\in V \text{ with } r(v)\in N \}.
	\end{align*}
\end{definition}

\begin{prop}
	\label{prop:propagation}
	Let $U\rightrightarrows M$ be a bisubmersion and $a\in \cE'_r(U)$.  For any $f\in C^\infty(M)$ we have
		\[
		\supp(\Op(a)f) \subseteq \supp(a)\circ\supp(f).
		\]
\end{prop}

\begin{proof}
	The support of $\Op(a)f$ lies in $r_U(\supp(a)\cap s_U^{-1}(\supp f)) = \supp(a)\circ\supp(f)$.
\end{proof}

\begin{definition}
	\label{def:proper}
	Let $U$ be a bisubmersion.  A subset $X\subseteq U$ is called \emph{proper} if it is both $r$- and $s$-proper.
\end{definition}

\begin{cor}
  \label{cor:proper_support_action}
  If $a\in \cE'_r(U)$ has proper support, then $\Op(a)$ maps $\Cc^\infty(M)$ into itself.
\end{cor}

\begin{proof}
\begin{sloppypar}
If $\supp(a)$ is $s$-proper and $\supp(f)$ is compact, then $\supp(a) \cap s_U^{-1}(\supp(f))$ is compact so $\supp(\Op(a)f)$ is compact.\qedhere
\end{sloppypar}
\end{proof}

The algebra $\cE'_s(\cF)$ of $s$-fibred distributions acts naturally on the distribution space $\cE'(M)$ via the transpose: if $b\in\cE'_s(U)$, we define $\widetilde{\Op}(b) \in \bL(\cE'(M))$ by
\[
 (\widetilde{\Op}(b)\omega, f) = (\omega, \Op(b^\rmt)f),
\]
for all $\omega\in\cE'(M)$, $f\in C^\infty(M)$.  It follows from Proposition \ref{prop:convolution_algebra} that this defines an algebra representation of $\cE'_s(\cF)$ on $\cE'(M)$.
Moreover, if $b\in\cE'_s(U)$ has $r$-proper support, then $\widetilde{\Op}$ extends to an action on $\cD'(M)$. 

Ultimately, we will want an algebra of distributions that acts on all four of the spaces $C^\infty(M)$, $\Cc^\infty(M)$, $\cE'(M)$ and $\cD'(M)$.  This requires the notion of transverse distributions, which we treat in Section \ref{sec:proper}.

\section{The ideal of smooth fibred densities}
\label{sec:smooth_ideal}

Inside the algebra of continuous operators on $C^\infty(M)$ is the right ideal of smoothing operators which map $\cD'(M)$ into $C^\infty(M)$.  This consists of Schwartz kernel operators with kernels in $C^\infty(M) \,\hat\otimes\, \Cc^\infty(M;|\Omega|)$.  The generalization of this in our context is the right ideal of smooth $r$-fibred densities.

Let $q:U \to M$ be a submersion. Let $|\Omega_q|$ denote the bundle of $1$-densities along the longitudinal tangent bundle of the $q$-fibration $\ker (dq) \subseteq TU$.  Let us write $C_q^\infty(U;|\Omega_q|)$ for the space of smooth sections of $|\Omega_q|$ with $q$-proper support.  Any such section $a$ defines a $q$-fibred distribution on $U$ by the formula
\[
 (a, \phi)(x) = \int_{u\in q^{-1}(x)} a(u) \phi(u).
\]
We call these elements \emph{smooth $q$-fibred densities}.  

In particular, if $(U,r,s)$ is a bisubmersion for $\cF$ then we have the subspaces $C_r^\infty(U;|\Omega_r|) \subset\cE'_r(U)$ and $C_s^\infty(U;|\Omega_s|) \subset\cE'_s(U)$.  

We note the following.

\begin{lemma}
 \label{lem:smooth_integration}
 Let $\pi:U' \to U$ be a submersive morphism of bisubmersions.
 Integration along the fibres restricts to a map
 \[
  \pi_* : C^\infty_r(U';|\Omega_r|) \to C^\infty_r(U;|\Omega_r|). 
 \] 
 Moreover, if $\pi$ is onto then this map is surjective.

 Analogous statements hold with $s$ in place of $r$ throughout.
\end{lemma}

\begin{proof}The only point we need to check is surjectivity. This follows from the lifting process in the proof of Lemma \ref{lem:integration_is_surjective}.
\end{proof}

\begin{definition}
We write 
\[
 C_r^\infty(\cU_\hol;|\Omega_r|) = \{ \bfa=(a_U) \in \cE'_r(\cU_\hol) : a_U \in C_r^\infty(U;|\Omega_r|) \text{ for all } U \in \cU_\hol\},
\]
and define $C_r^\infty(\cF;|\Omega_r|)$ to be the image of $C_r^\infty(\cU_\hol;|\Omega_r|)$ in the quotient $\cE'_r(\cF)$.

We define $C_s^\infty(\cF) \subseteq \cE'_s(\cF)$ similarly.
\end{definition}

The main result of this section is the following.

\begin{thm}
 \label{thm:smooth_ideal}
 The space $C^\infty_r(\cF;|\Omega_r|)$ is a right ideal in the algebra $\cE'_r(\cF)$.  Likewise, $C^\infty_s(\cF;|\Omega_s|)$ is a left ideal in $\cE'_s(\cF)$.
\end{thm}

The rest of this section is dedicated to proving Theorem \ref{thm:smooth_ideal} for $C^\infty_s(\cF;|\Omega_s|)$.  The result for $C^\infty_r(\cF;|\Omega_r|)$ follows from this by applying the transpose and using Proposition \ref{prop:convolution_algebra}(b).  

\medskip

By linearity, it suffices to treat the case where $\bfa$ and $\bfb$ each live on a single bisubmersion, i.e., taking $a\in\cE'_s(U)$ and $b\in\cE'_s(V)$ for some bisubmersions $U$ and $V$.

We begin with the special case where $U=V$ is a minimal path holonomy bisubmersion at some point $x_0\in M$.  Therefore, let $\bfX = (X_1, \ldots , X_m)$ be a minimal generating family at $x_0\in M$ and let $U$ be an associated path holonomy bisubmersion.  This means we fix a small enough neighbourhood $M_0 \subseteq M$ of $x_0$ and a small enough neighbourhood $U \subseteq \R^m \times M_0$ of $\{0\} \times M_0$ so that we can define source and range maps on $U$ by
\[
 s(\xi,x) = x, \qquad r(\xi,x) = \Exp_\bfX(\xi,x),
\]
see Definition \ref{def:path_holonomy_bisubmersion} for notation.

By Proposition 2.7 of \cite{AS2}, there is a neighbourhood $U'\subseteq U$ of $(0,x_0)$ which admits a submersive morphism of bisubmersions
\begin{equation*}
 \label{eq:nice_neighbourhood}
 \pi : U'\circ U' \to U
\end{equation*}
with $\pi((0,x),(0,x)) = (0,x)$ for all $x\in M_0$.
Recall that
\[
 U'\circ U' = \{(\eta,y),(\xi,x) \in U'\times U' : y = \Exp_\bfX(\xi,x) \}.
\]
The coordinate $y$ is superfluous, so let us henceforth make the identification
\[
 U'\circ U' = \{(\eta,\xi,x) \in \R^m\times\R^m\times M_0 : 
                 (\xi,x)\in U' \text{ and } (\eta,\Exp_\bfX(\xi,x)) \in U' \},
\]
with range and source maps
\[
 s(\eta,\xi,x) = x, \qquad r(\eta,\xi,x) = \Exp_\bfX(\eta,\Exp_\bfX(\xi,x)).
\]

Let $\pi_1: U'\circ U' \to \RR^m$ be the function determined by
\[
 \pi(\eta,\xi,x) = (\pi_1(\eta,\xi,x),x).
\]
If we fix $\xi=0$ then we have $\pi_1(\eta,0,x) = \eta$ for all $(\eta,0,x)\in U'$, so the derivative $D_\eta\pi_1$ of $\pi_1$ with respect to the $\eta$ variables is the identity at all $(\eta,0,x)\in U'$.  Therefore, the map
\begin{align*}
 \Pi:U'\circ U' &\to \R^m \times \R^m \times M_0 \\
 \Pi(\eta,\xi,x) &= (\pi_1(\eta,\xi,x), \xi, x),
\end{align*}
has invertible derivative at every $(\eta,0,x)\in U'$.  By further restricting the neighbourhood $U'$ of $(0,x)$ in $U$, the map $\Pi$ is a diffeomorphism onto its image.  We thus have a smooth function $\theta:\R^m \times \R^m \times M_0 \to \R^m$ such that
\[
 \Pi(\theta(\eta,\xi,x),\xi,x) = (\eta, \xi, x)
\]
for all $(\eta, \xi, x)$ in $\Pi(U'\circ U')$, or equivalently
\[
  \pi(\theta(\eta,\xi,x),\xi,x) = (\eta, x).
\]

\begin{lemma}
 \label{lem:smooth_ideal1}
 Let $U$ be a minimal path holonomy bisubmersion at $x_0\in M$.  With the above notation, there exists a neighbourhood $U'$ of $(0,x_0)$ in $U$ such that the map $\Pi$ is a diffeomorphism onto its image. Then for any $a\in C_s^\infty(U';|\Omega_s|)$, and $b\in\cE'_s(U')$ we have
 \[
  \pi_*(a*b) \in C_s^\infty(U;|\Omega_s|).
 \]
\end{lemma}

\begin{proof}
 The existence of the neighbourhood $U'$ was proven in the discussion preceding the lemma.  
 Let $\phi\in C^\infty(U'\circ U')$.  We have
 \[
  (\pi_*(a*b) , \phi) = (b\circ r^*a, \pi^*\phi).
 \]
 Let us write $a = a_0(\xi,x) d\xi$ where $a_0$ is some $s$-properly supported smooth function on $U'$ and $d\xi$ is Lebesgue measure on $\R^n$.  For every $(\eta,\xi,x) \in U'\circ U'$ we have $\pi(\eta,\xi,x) = (\pi_1(\eta,\xi,x),x)$, so that
 \begin{align*}
  (r^*a_0,\pi^*\phi)(\xi,x) 
   &=\int_{\eta\in\R^m} r^*a_0(\eta,\xi,x) \phi(\pi_1(\eta,\xi,x),x) \,d\eta \\
   &=\int_{\eta\in\R^m} r^*a_0(\theta(\eta,\xi,x),\xi,x) \phi(\eta,x) 
    |D_\eta\theta(\eta,\xi,x)| \, d\eta,
 \end{align*}
 where in the last line we have used the change of variables $\eta \to \theta(\eta,\xi,x)$.  
 Let us write 
 \[
  \tilde{a}_0(\eta,\xi,x) 
    = r^*a_0(\theta(\eta,\xi,x),\xi,x) |D_\eta\theta(\eta,\xi,x)|,
 \]
 which is a smooth function on $\Pi( U'\circ U') \subseteq \R^m \times \R^m \times M_0$.  Using the formal notation $\int_\xi b(\xi,x)\psi(\xi,x)\,d\xi$ to denote the value of $(b,\psi)(x)$, Fubini's Theorem for distributions gives
 \begin{align*}
  (\pi_*(a*b) , \phi)(x) 
   &= \int_{\xi \in \R^m} b(\xi,x) 
       \left( \int_{\eta\in\RR^m} \tilde{a}_0(\eta,\xi,x) \phi(\eta,x) \,d\eta \right) d\xi \\
   &= \int_{\eta\in\RR^m} \left( \int_{\xi \in \R^m} b(\xi,x) \tilde{a}_0(\eta,\xi,x) \, d\xi \right)
       \phi(\eta,x) \,d\eta.
 \end{align*}
 The integral in brackets in the last line is a smooth function of $(\eta,x)\in\R^m\times M_0$. Since $\pi_*(a*b)$ is automatically $s$-properly supported, the result follows.
\end{proof}

\begin{lemma}
 \label{lem:smooth_ideal2}
 Let $V, W \in \cU_\hol$ and suppose $S\subseteq V$ and $T\subset W$ are local identity bisections.  For every $v\in S$ and $w\in T$ there exist open neighbourhoods $V_v\subseteq V$ of $v$ (depending only on $v$) and $W_{v,w} \subseteq W$ of $w$ (depending on both $v$ and $w$) such that for any $a\in C_s^\infty(V_v;|\Omega_s|)$ and $b\in\cE'_s(W_{v,w})$ we have $a*b \in C_s^\infty(\cF;|\Omega_s|)$.
\end{lemma}

\begin{proof}
 Let us put $x_0 = s_V(v)= r_V(v)$.  Let $U$ be a minimal path-holonomy bisubmersion at $x_0$ and let $U'\subseteq U$ be a neighbourhood of $(0,x_0)$ of the kind described in Lemma \ref{lem:smooth_ideal1}.  Since $S$ is an identity bisection at $v$, Proposition 2.10 of \cite{AS1} shows that we can find a submersive morphism $\pi_V: V_v \to U'$ with $\pi_V(v) = (0;x_0)$ for some neighbourhood $V_v$ of $v$.  After possibly reducing $V_v$, we may assume that the closure of $s_V(V_v)$ lies in $s_U(U')$.

 Now put $y_0 = s_W(w) = r_W(w)$.  To define $W_{v,w}$, we consider two cases:
 \begin{itemize}
  \item 
  Suppose $y_0 \notin s_U(U')$.  Then we can find a neighbourhood $W_{v,w}$ of $w$ such that $r_W(W_{v,w}) \cap s_V(V_v) = \emptyset$.  In this case, Lemma \ref{lem:disjoint_support} shows that for all $a\in C_s^\infty(V_v;|\Omega_s|)$ and $b\in\cE'_s(W_{v,w})$ we have $a*b=0$, which proves the claim.

  \item
  Suppose $y_0\in s_U(U')$.  Let $\widetilde{U}$ be a minimal path-holonomy bisubmersion at $y_0$.  Given that we have identity bisubmersions $T \subset W$ passing through $w$ and $\{0\} \times M_0 \subset U$ passing through $(0,y_0)$, Proposition 2.10 of \cite{AS1} shows that there exist submersive morphisms
  \begin{align*}
   \pi_W : W_{v,w} &\to \widetilde{U}  &&\text{with } \pi_W(w) = (0,y_0), \\
   \pi_U : U_{v,w} & \to \widetilde{U}  &&\text{with } \pi_U(0,y_0) = (0,y_0),
  \end{align*}
  for some neighbourhoods $W_{v,w} \subseteq W$ of $w$ and $U_{v,w} \subseteq U'$ of $(0;y_0)$.  Moreover, by reducing $\widetilde{U}$ sufficiently, we may assume that $\pi_U$ is surjective.  Now, given  $a\in C_s^\infty(V_v;|\Omega_s|)$ and $b\in\cE'_s(W_{v,w})$, Lemma \ref{lem:integration_is_surjective} shows that we can find $\widetilde{b} \in \cE'_s(U_{v,w})$ such that $\pi_{U*}\widetilde{b} = \pi_{W*}b$. Then Lemma \ref{lem:smooth_ideal1} shows that
  \[
   a*b \equiv \pi_{V*}a * \widetilde{b} \in C_s^\infty(\cF;|\Omega_s|),
  \]
  which again proves the claim.
 \end{itemize}
\end{proof}

In order to extend this result to general bisubmersions, we use translations by bisections. 

\begin{definition}
 \label{def:bisubmersion_translation}
 Let $U, V \in \cU_\hol$ be bisubmersions.  Let $S \subset V$ be a local bisection, and let $\Phi_S=r|_S\circ s|_S^{-1}$ be the local diffeomorphism induced by $S$.  We define the \emph{right-translate} of $(U,r_U,s_U)$ by $S$ to be $(U_S,r_{U_S},s_{U_S})$ with
 \[
  U_S = s_U^{-1}(r_V(S)) \subseteq U, \qquad r_{U_S} = r_U, \qquad s_{U_S} = \Phi_S^{-1}\circ s_U.
 \]
 Likewise, the \emph{left-translate} of $(U,r_U,s_U)$ by $S$ is $(U^S,r_{U^S},s_{U^S})$ with
 \[
  U^S = r_U^{-1}(s_V(S)) \subseteq U, \qquad r_{U^S} = \Phi_S\circ r_U, \qquad s_{U_S} = s_U.
 \]
\end{definition}

Using the fact that the local diffeomorphism $\Phi_S$ preserves the foliation $\cF$, we see that $U_S$ and $U^S$ satisfy the bisubmersion axioms.  We want to further check that they belong to the maximal path-holonomy atlas.

For this, consider the subset $ U \circ S = U \srtimes S $ of $U \circ V$.  
The projection $ \pr_U: U \circ S \to U$ is a diffeomorphism onto $U_S$ which intertwines the range and source maps, and it follows that the embedding
\begin{equation}
 \label{eq:right_translation_inclusion}
 \iota_S : U_S \stackrel{\pr_U^{-1}}{\longrightarrow} U \circ S \hookrightarrow U\circ V
\end{equation}
is a morphism of bisubmersions.  This proves that $U_S$ is adapted to $\cU_\hol$.  The proof for $U^S$ is similar, using the morphism of bisubmersions
\begin{equation}
 \label{eq:left_translation_inclusion}
 \iota^S : U^S \stackrel{\pr_U^{-1}}{\longrightarrow} S \circ U \hookrightarrow V\circ U
\end{equation}
defined in the analogous way.

Next, we want to introduce the left and right translates of an $s$-fibred distribution.

We note that on the open subset $U_S\subseteq U$ the fibres of the two source maps $s_{U}$ and $s_{U_S} = \Phi_S^{-1}\circ s_U$ are the same, although the maps themselves are different.  Therefore, if $a\in \cE_s'(U)$ has $\supp(a)\subseteq U_S$, then $a$ also defines an $s$-fibred distribution on $U_S$, which we denote by $a_S$ and call the \emph{right translate of $a$ by $S$}.  Explicitly, $a_S$ is determined by
\[
 (a_S,\phi|_{U_S})(x) = 
  \begin{cases}
   (a,\phi)(\Phi_S(x)), & \text{if }x\in s_V(S) \\
   0, & \text{otherwise},
  \end{cases}
\]
for all $\phi\in C^\infty(U)$.  

A similar definition can be made for the left translate, and in fact is even easier since $s_{U^S} = s_U$ on $U^S$.  In this case, for $b\in \cE'_s(U)$ with $\supp(b)\subseteq U^S$ we define
\[
 (b^S, \phi|_{U^S}) = (b, \phi),
\]
for all $\phi\in C^\infty(U)$.

We collect some basic facts about left and right translates.

\begin{lemma}
 \label{lem:translates}
 Let $U$, $V$, $W$ be bisubmersions and let $S\subset V$ be local bisections.  If $a\in\cE'_s(U)$ with $\supp(a)\subseteq U_S$ then:
 \begin{enumerate}
  \item We have $(W\circ U)_S = W\circ U_S$, and for any $b\in\cE'_s(W)$ we have $(b*a)_S = b*a_S$.
  \item The right translate $a_S$ is a smooth $s$-fibred density if and only if $a$ is.
 \pauseenumerate
 \end{enumerate}
 Similarly, if $a\in\cE'_s(U)$ with $\supp(a) \subseteq U^S$ then:
 \begin{enumerate}
 \resumeenumerate
 	\item We have $(U \circ W)^S = U^S \cap W$, and for any $b\in\cE'_s(W)$ we have $(a*b)^S = a^S*b$.
  \item The left translate $a^S$ is a smooth $s$-fibred density if and only if $a$ is.
 \end{enumerate}
\end{lemma}

\begin{proof}
 For (a), note that
 \[
  (W\circ U)_S = \{(w,u) \in W\srtimes U : s_U(u) \in r_V(S)\}
   = W\circ U_S.
 \]
 Thus, the right translate $(b*a)_S$ makes sense, and its restrictions to the $s$-fibres of $W\circ U_S$ are identical to those of $b*a$, and hence to those of $b*a_S$.
 Part (b) follows immediately from the fact that the restrictions of $a$ and $a_S$ to the fibres of $U_S$ are identical.
 
 The other two statements are analogous.
\end{proof}

\begin{lemma}
 \label{lem:smooth_ideal3}
 Let $V, W \in \cU_\hol$.  For every $v\in V$ and $w\in W$ there exist open neighbourhoods $V_v\subseteq V$ of $v$ (depending only on $v$) and $W_{v,w} \subseteq W$ of $w$ (depending on both $v$ and $w$) such that for any $a\in C_s^\infty(V_v;|\Omega_s|)$ and $b\in\cE'_s(W_{v,w})$ we have $a*b \in C_s^\infty(\cF;|\Omega_s|)$.
\end{lemma}

\begin{proof}
 Let $S \subset V$ be a local bisection passing through $v$.  Note that $S^\rmt$ is a local bisection of $V^\rmt$, so we may consider the left translate $V^{S^\rmt}$ of $V$ by $S^\rmt$.
 In this bisubmersion, the set $S$ is an identity bisection, since
 \[
  r_{V^{S^\rmt}}(z) = \Phi_{S^\rmt}\circ r_V(z) = s_V(z) = s_{V^{S^\rmt}}(z)
 \]
 for all $z\in S$.  
 
 Likewise, if $T\subset W$ is a bisection passing through $w$, then $T$ is an identity bisection in the right translate $W_{T^\rmt}$.
 
 Therefore, by Lemma \ref{lem:smooth_ideal2}, we can find neighbourhoods $V_v$ of $v$ in $V^{S^\rmt}$ and $W_{v,w}$ of $w$ in $W_{T^\rmt}$ verifying the conditions of Lemma \ref{lem:smooth_ideal2}.  

 Let $a\in C_s^\infty(V;|\Omega_s|)$ with $\supp(a) \subseteq V_v$ and $b\in\cE'_s(W)$ with $\supp(b) \subseteq W_{v,w}$.  By Lemmas \ref{lem:smooth_ideal2} and \ref{lem:translates} (a) and (c) we have 
 \[
  a^{S^\rmt} * b_{T^\rmt} = ((a*b)^{S^\rmt})_{T^\rmt} \in C_s^\infty(\cF;|\Omega_s|).
 \]

By Lemma \ref{lem:translates} (b) and (d), this implies $a*b \in C_s^\infty(\cF;|\Omega_s|)$, as claimed.
\end{proof}

\begin{proof}[Proof of Theorem \ref{thm:smooth_ideal}]
 Let $a\in C^\infty_s(V;|\Omega_s|)$ and $b\in \cE'_s(W)$ for some $V,W \in \cU_\hol$.  

 For each $v\in V$, pick an open neighbourhood $V_v$ of $v$ as in Lemma \ref{lem:smooth_ideal3}.  Using a smooth locally finite partition of unity subordinate to the cover $(V_v)_{v\in V}$, we can reduce to the case where $a$ is supported on $V_v$ for some $v\in V$.

 Next, for each $w\in W$, pick an open neighbourhood $W_{v,w}$ of $w$ as in Lemma \ref{lem:smooth_ideal3}.  Again, using a partition of unity we can reduce to the case where $b$ is supported in $W_{v,w}$ for some $w\in W$.  Then Lemma \ref{lem:smooth_ideal3} completes the proof.
\end{proof}

\section{Transverse distributions}
\label{sec:proper}

We now want to consider Schwartz kernels
which are both $r$- and $s$-fibred.  As defined above, the spaces $\cE'_r(U)$ and $\cE'_s(U)$ are not comparable, so we must put both of them into the usual space of distributions $\cD'(U)$. 

Throughout this section, we will fix a choice of nowhere-vanishing smooth $1$-density on the base space, $\mu\in C^\infty(M;|\Omega|)$.  As we will show, the particular choice of $\mu$ ultimately will not affect the results.

\begin{definition}
	Let $U$ be a bisubmersion and $\mu$ a nowhere-vanishing smooth density on $M$.  We define maps
	\begin{align*}
	 \mu_r : \cE'_r(U) &\hookrightarrow \cD'(U); & \mu_r(a) = \mu \circ a,\\
	 \mu_s : \cE'_s(U) &\hookrightarrow \cD'(U); & \mu_s(b) = \mu \circ b.
	\end{align*}
\end{definition}

The maps $\mu_r$ and $\mu_s$ are continuous injective linear maps. Notice that $\supp \mu_{r}(a) = \supp(a)$ and $\supp \mu_{s}(b) = \supp(b)$. The image $\mu_r(\cE'_r(U)) \subseteq \cD'(U)$ coincides with the space of distributions (with $r$-proper support) transverse to the submersion $r$, as defined in \cite[\S{}1.2.1]{AS2}; see also \cite{LMV}.  This shows that the image is independent of the choice of smooth density $\mu$.  Analogous statements hold for $\mu_s$, of course.

\begin{definition}
 Let $U$ be a bisubmersion.  
 \begin{enumerate}
	\item A distribution in $\cD'(U)$ is called \emph{transverse} if it belongs to $\mu_r(\cE'_r(U))\cap\mu_s(\cE'_s(U))$.  The space of transverse distributions on $U$ is denoted $\Dp'(U)$.
	\item An $r$-fibred distribution $a \in \cE'(U)$ is called \emph{trasnverse} if $\mu_r(a)$ is transverse.  The space of proper $r$-fibred distributions on $U$ is denoted $\cE'_{r,s}(U)$.
	\item An $s$-fibred distribution $b \in \cE'(U)$ is called \emph{transverse} if $\mu_s(b)$ is transverse.  The space of transverse $s$-fibred distributions on $U$ is denoted $\cE'_{s,r}(U)$.
 \end{enumerate}
\end{definition}

Since $r$-fibred distributions have $r$-proper support, and $s$-fibred distributions have $s$-proper support, we see that transverse distributions of any kind have proper support in the sense of Definition \ref{def:proper}.

Let us write $\Cp^\infty(U)$ for the space of properly supported functions on $U$ and $\Cp^\infty(U;E)$ for the space of properly supported sections of any bundle $E$ over $U$.
Note that
\[
 \mu_r(\Cp^\infty(U;|\Omega_r|)) = \Cp^\infty(U;|\Omega|) = \mu_s(\Cp^\infty(U;|\Omega_s|)),
\]
so that properly supported smooth $r$-fibred densities are automatically transverse as $r$-fibred distributions, and likewise for properly supported smooth $s$-fibred densities.

\begin{lemma}
	\label{lem:proper_convolution}
	Let $U$ and $V$ be bisubmersions and let $a\in \cE'_{r,s}(U)$ and $b\in \cE'_{r,s}(U)$ be transverse $r$-fibred distributions, so that there exist $\tilde{a}, \tilde{b}\in \cE'_s(U)$ with $\mu_r(a)=\mu_s(\tilde{a})$ and $\mu_r(b) = \mu_s(\tilde{b})$.  Then
	\[
	 \mu_r(a*b) = \mu_s(\tilde{a}*\tilde{b}).
	\]
	In particular, the convolution product of two transverse $r$-fibred distributions is again transverse, and likewise for transverse $s$-fibred distributions.
\end{lemma}

\begin{proof}
	We first claim that for any $\tilde{a}\in\cE'_s(U)$ and $b\in\cE'_r(V)$ we have $\tilde{a}\circ s_U^*b = b\circ r_V^*\tilde{a}$ as maps $C^\infty(U\circ V) \to C^\infty(M)$.  
	Note that both these maps are $C^\infty(M)$-linear with respect to the `middle' submersion, 
	\begin{align*}
	  q:U\circ V \to M; \qquad q(u,v) = s_U(u) = r_V(V),
	\end{align*}
	that is, we have
	\[
	 (\tilde{a}\circ s_U^*b,(q^*f)\phi) = f(\tilde{a}\circ s_U^*b,\phi)
	\]
	for all $\phi\in C^\infty(U\circ V)$ and $f\in C^\infty(M)$, and similarly for $b\circ r_V^*\tilde{a}$.
	The $q$-fibre at $x\in M$ is $q^{-1}(x) = U_x \times V^x$, and we calculate
	\[
	 (\tilde{a}\circ s_U^*b, \phi)(x) 
	  = (\tilde{a}_x \otimes b^x, \phi|_{U_x \times V^x})
	  = (b\circ r_V^*\tilde{a}, \phi)(x),
	\]
	which proves the claim.
	
	Finally, we obtain
	\begin{align*}
	 \mu_r(a*b) 
	  &= \mu\circ a \circ s_U^*b
	  = \mu\circ\tilde{a}\circ s_U^*b \\
	  &\qquad = \mu\circ b\circ r_V^*\tilde{a}
	  = \mu\circ \tilde{b}\circ r_V^*\tilde{a}
	  = \mu_s(\tilde{a}\circ\tilde{b}).
	\end{align*}
\end{proof}

\begin{definition}
	We define $\cE'_{r,s}(\cF)$ to be the subspace of $\cE'_r(\cF)$ consisting of classes of locally finite sums of transverse $r$-fibred distributions on bisumbmersions in $\cU_\hol$, and $\cE'_{s,r}(\cF)$ as the subspace of $\cE'_s(\cF)$ consisting of classes of  locally finite sums of transverse $s$-fibred distributions on bisumbmersions in $\cU_\hol$.
\end{definition}

\begin{ex}
	The pseudodifferential kernels on a singular foliation $\cF$ defined in \cite{AS2} also define elements of $\cE'_{r,s}(\cF)$.  This follows from Remark \ref{rmk:PsiDOs} and \cite[Proposition 1.3]{AS2}.
\end{ex}

\begin{prop}
	The subspace $\cE'_{r,s}(\cF)$ is a closed subalgebra of $\cE'_r(\cF)$ and $\Cp^\infty(\cF;|\Omega_r|)$ is a two-sided ideal in $\cE'_{r,s}(\cF)$.  
	
	Similarly for $\cE'_{s,r}(\cF)$ and $\Cp^\infty(\cF;|\Omega_s|)$.
\end{prop}

\begin{proof}
 The fact that $\cE'_{r,s}(\cF)$ is closed under convolution follows from Lemma \ref{lem:proper_convolution}.  Theorem \ref{thm:smooth_ideal} shows immediately that $\Cp^\infty(\cF;|\Omega_r|)$ is a right ideal, and also that it is a left ideal if we take advantage of Lemma \ref{lem:proper_convolution}.
\end{proof}

We note the relationship of the transpose map from Definition \ref{def:convolution} with the maps $\mu_r$ and $\mu_s$:
\[
 \mu_r(a)^\rmt = \mu_s(a^\rmt)
\]
It follows that the transpose of a transverse $r$-fibred distribution is a transverse $s$-fibred distribution, and so by Proposition \ref{prop:convolution_algebra}, transposition gives an anti-isomorphism of convolution algebras ${}^\rmt\!:\cE'_{r,s}(\cF) \to \cE'_{s,r}(\cF)$.

We finish this section by recalling that transversality is a consequence of standard wave-front conditions due to H\"ormander; see \cite{Hormander:Volume1}.  For instance, a proof of the following result can be found \cite{LMV}.

\begin{prop}[\cite{Hormander:Volume1,LMV}]
	Let $S$ be a bisection in a bisubmersion $U$, and suppose that $a\in\cD'(U)$ is a distribution with
	\begin{itemize}
		\item proper support,
		\item singular support contained in $S$,
		\item wavefront set contained in $TS^\perp = \{\eta \in T^*U : (\eta,\xi) = 0 \text{ for } \xi \in TS \}$.
	\end{itemize}
	Then $a$ is a transverse distribution.
\end{prop}

\section{The action on generalized functions on the base}\label{sec:actgen}

According to Section \ref{sec:Op}, the properly supported $r$-fibred distributions act on $C^\infty(M)$ and $\Cc^\infty(M)$, while properly supported $s$-fibred distributions act on $\cE'(M)$ and $\cD'(M)$.  If we fix a nowhere vanishing smooth density $\mu$ on $M$, then by Lemma \ref{lem:proper_convolution} we obtain an algebra isomorphism $\mu_s^{-1}\circ\mu_r : \cE'_{r,s}(U) \to \cE'_{s,r}(U)$, and hence an action of $\cE'_{r,s}(\cF)$ on all four of the above spaces.  However, the action of $\cE'_{r,s}(\cF)$ on $\cD'(M)$ which is obtained in this way depends upon the choice of $\mu$.  To obtain a canonical action, we need to work with generalized functions instead of distributions.

\begin{remark}
	This issue wouldn't arise if we had followed the operator algebraists' strategy of using half-densities throughout.   In that case, our convolution algebra would act on the space of half-densities on $M$ instead of functions. The two approaches are of course equivalent once one fixes a smooth density on $M$.
\end{remark}

\begin{definition}
	Let $|\Omega|$ denote the bundle of $1$-densities on $M$.  We write $C^{-\infty}(M)$ for the continuous linear dual of $\Cc^\infty(M;|\Omega|)$, and refer to its elements as \emph{generalized functions} on $M$.  We also write $\Cc^{-\infty}(M) = C^\infty(M;|\Omega|)^*$ for the compactly supported generalized functions.
\end{definition}

The space of smooth functions $C^\infty(M)$ admits a canonical embedding as a dense subspace of $C^{-\infty}(M)$.  
Given a choice of nowhere vanishing $1$-density $\mu$ on $M$, we obtain a linear isomorphism
\[
  C^\infty(M) \to C^\infty(M;|\Omega|); 
  \qquad f \mapsto f\mu,
\]
and this extends by density to an isomorphism
\[
  C^{-\infty}(M) \to \cD'(M)
\]
which we denote formally by $k \mapsto k\mu$ for $k\in C^{-\infty}(M)$.

\begin{definition}
 \label{def:Op2}
 Fix a nowhere-vanishing smooth density $\mu$ on $M$.  Let $a\in \cE'_{r,s}(U)$ be a transverse $r$-fibred density on the bisubmersion $U$, which means there is $\tilde{a}\in\cE'_{s,r}(U)$ with $\mu_s(\tilde{a}) = \mu_r(a)$.  We define the operator $\Op(a)$ on $C^{-\infty}(M)$ by the formula
 \[
  (\Op(a) k)\mu = (\widetilde{\Op}(\tilde{a}))(k\mu) 
 \]
 where $k\in C^{-\infty}(U)$ and $\widetilde{\Op}$ is the representation of $\cE'_{s,r}(U)$ on $\cD'(M)$ defined at the end of Section \ref{sec:Op}.
\end{definition}

\begin{prop}
 \label{prop:generalized_functions}
 The map $\Op:\cE'_{r,s}(\cF) \to \bL(C^{-\infty}(M))$ defined above is independent of the choice of smooth nowhere-vanishing density $\mu$.  It induces a continuous linear representation of $\cE'_{r,s}(\cF)$ on generalized functions which extends the representation $\Op$ on $C^\infty(M)$. 
\end{prop}

\begin{proof}
 Let $k\in C^\infty(M)$.  If we use the definition of $\Op(a)k$ from Definition \ref{def:Op2}, we get, for all $f\in C^\infty(M)$,
 \begin{align*}
  ((\Op(a)k)\mu, f)
   &= (\widetilde{\Op}(\tilde{a})(k\mu), f) \\
   &= (k\mu, \Op(\tilde{a}^\rmt)f) \\
   &= \int_{x\in M} (\tilde{a}^\rmt,r_U^*f)(x) k(x) \mu(x) \\
   &= \int_{x\in M} (\tilde{a}^\rmt,(r_U^*f)(s_U^*k))(x) \mu(x) \\
   &= (\mu_s(\tilde{a}), (r_U^*f)(s_U^*k)),
 \end{align*}
 where the second last equality uses the $C^\infty(M)$-linearity of $\tilde{a}^t$ as an $r$-fibred distribution.
 On the other hand, using the definition of $\Op(a)k$ from Propostion \ref{prop:Op} gives
 \begin{align*}
  ((\Op(a)k)\mu, f)
   &= \int_{x\in M} (a,s_U^*k)(x) \mu(x) f(x) \\
   &= \int_{x\in M} (a,(r_U^*f)(s_U^*k))(x) \mu(x) \\
   &= (\mu_s(\tilde{a}), (r_U^*f)(s_U^*k)),
 \end{align*}
 where the second equality uses the $C^\infty(M)$-linearity of $a$ as an $r$-fibred distribution.  This proves that the two definitions of $\Op(a)k$ agree when $k\in C^\infty(M)$.  The definition of $\Op(a)k$ from Proposition \ref{prop:Op} clearly does not depend on the choice of $\mu$, so neither does Definition \ref{def:Op2} in this case.  By density, the same must be true for all $k\in C^{-\infty}(M)$.

 For $a\in \cE'_{r,s}(U)$ and $b\in \cE'_{r,s}(V)$, Lemma \ref{lem:proper_convolution} plus the fact that $\widetilde{\Op}$ is an algebra representation of $\cE_{s,r}(\cF)$ gives
	\begin{align*}
	 (\Op(a*b)k)\mu = \widetilde{\Op}(\tilde{a}*\tilde{b})(k\mu)
	   = \widetilde{\Op}(\tilde{a}) \widetilde{\Op}(\tilde{b}) (k\mu)
	   = (\Op(a)\Op(b)k)\mu,
	\end{align*}
 so $\Op$ is indeed a representation.  
\end{proof}

Therefore, representation $\Op$ defines an action of $\cE'_{r,s}(\cF)$ on each of the four spaces $C^\infty(M)$, $\Cc^\infty(M)$, $C^{-\infty}(M)$ and $\Cc^{-\infty}(M)$.

\section{Action on the leaves and their holonomy covers}
 \label{sec:actleaf}
 \label{sec:actholcov}

We conclude by describing the natural actions of the algebra $\cE'_{r,s}(\cF)$ on the spaces of functions and generalized functions on a leaf $L$ of the foliation and its holonomy cover $\tilde{L}$.

Recall that $H(\cF)$ denotes the holonomy groupoid of $(M,\cF)$, see Section \ref{sec:holonomy_groupoid}. Let $L=L_x$ be the leaf through $x \in M$.
Note $L$ is a manifold in its own right, and we equip it with its manifold topology and not its topology as a subspace of $M$.  We write $\iota_L : L \to M$ for the inclusion.
Then 
\[
 H(\cF)_L := L\;{}_{\iota_L}\!\!\times_r H(\cF)
\]
is the fibre of the holonomy groupoid over the leaf $L$.  We have a natural identification $H(\cF)_L = r^{-1}(L) = s^{-1}(L)$ as a set, but the definition as a fibred product gives it a manifold topology.  In fact, from
\cite{Debord2013}, we know that $H(\cF)_L$ is a transitive Lie groupoid with base $L$.

Now let $U$ be a bisubmersion in the path holonomy atlas.  We write
\[
U_L = L\;{}_{\iota_L}\!\!\times_{r_U} U,
\]
which, similarly, can be identified with $s_U^{-1}(L) = r_U^{-1}(L)$ as a set, but has a manifold topology.  The
quotient map $q_U:U \to H(\cF)$ induces a submersive map between manifolds $q_{U_L} = \id \times q_U:U_L \to H(\cF)_L$, which commutes with the $r$ and $s$ maps on $U_L$ and $H(\cF)_L$, respectively.

	\begin{definition}
		For $a\in\cE'_r(U)$, we define the \emph{restriction of $a$ to $H(\cF)_L$} by
		\[
		a|_{H(\cF)_L} := {q_{U_L}}_* (\iota_L^* a) \in \cE'_r(H(\cF)_L).
		\]
	\end{definition}
	
The pullback and pushforward operations here are as in Section \ref{sec:fibdistrsubm}
	
	\begin{prop}
		\label{prop:a_L}
		Let $a\in\cE'_r(U)$ and $b\in\cE'_r(V)$ be $r$-fibred distributions on a pair of bisubmersions.
		\begin{enumerate}
			\item If $\pi:U \to U'$ is a morphism of bisubmersions, then $(\pi_*a)|_{H(\cF)_L} = a|_{H(\cF)_L}$.  
			\item We have $(a*b)|_{H(\cF)_L} = a|_{H(\cF)_L} * b|_{H(\cF)_L}$.
		\end{enumerate}
		Consequently, the family of maps $a \mapsto a|_{H(\cF)_L}$ on each bisubmersion yield a well-defined topological algebra map $\cE'_r(\cF) \to \cE'_r(H(\cF)_L)$, where the latter is the convolution algebra of Lescure-Manchon-Vassout \cite{LMV} for the Lie groupoid $H(\cF)_L$.  Moreover, this map sends $C^\infty_r(\cF;|\Omega_r|)$ to $C^\infty_r(H(\cF)_L;|\Omega_r|)$.
	\end{prop}
	
	\begin{proof}
		For a), note that the morphism $\pi$ induces a map
		\[
		\pi_L = \id\times\pi : L\;{}_{\iota_L}\!\!\times_{r_U} U \to L\;{}_{\iota_L}\!\!\times_{r_{U'}} {U'}.
		\]
		Then Lemma \ref{lem:pullback-functoriality} a) gives
		\[
		(\pi_*a)|_{H(\cF)_L} 
		= {q_{U'_L}}_* (\iota_L^* \pi_*a)
		= {q_{U'_L}}_* (\id\times\pi)_* (\iota_{L}^* a)
		= {q_{U_L}}_* (\iota_{L}^* a) 
		= a|_{H(\cF)_L}.
		\]
		
		For b), consider the commuting diagram
		\[
		\xymatrix@!C=16mm{
			U_L \ar[r]^{\iota_{U_L}} \ar[d]_{s_U} &
			U \ar[d]^{s_U} \\
			L \ar[r]_{\iota_L} &
			M.
		}
		\]
		We can pullback the submersion $r_V:V \to M$ by each of these maps, and we obtain the commuting diagram
		\[
		\xymatrix@!C=16mm{
			U_L \;{}_r\!\! \times_s V_L \ar[r] \ar[d] &
			U \;{}_r\!\! \times_s V \ar[d] \\
			V_L \ar[r] &
			V,
		}
		\]
		with submersions from each of the spaces in the second square the corresponding ones in the first.  From this, we obtain the equality of pullbacks of $r$-fibred distributions
		\[
		\iota_L^* (a*b) = \iota_L^* (a\circ s_U^*b) = (\iota_{L}^*a)\circ ({s_U}_*({\iota_L}_*b)) = (\iota_{L}^*a)* ({\iota_L}_*b).
		\]  
		Applying $(q_{U\circ V})_*$ gives the result.
		
		The final claims follow immediately.
	\end{proof}

	For any Lie groupoid $G\rightrightarrows G^{(0)}$, the algebra $\cE'_r(G)$ of Lescure-Manchon-Vassout \cite{LMV} admits canonical actions on the space of smooth functions $C^\infty(G^{(0)})$ on the base, and also on $C^\infty(G_x)$ for every $s$-fibre $G_x$, with $x\in M$.  Therefore, using the algebra morphism $\cE'_{r}(\cF) \to \cE'_r(H(\cF)_L)$ from Proposition \ref{prop:a_L}, we immediately obtain an  action $\Op_L$ of our singular algebra $\cE'_r(\cF)$ on smooth functions on the leaf $L=H(\cF)_L^{(0)}$ and an action $\Op_{\tilde{L}}$ on the smooth functions on its holonomy cover $\tilde{L} \cong H(\cF)_x$ for any $x\in L$.
	
	Explicitly, the action $\Op_L(a)$ of an $r$-fibred distribution $a\in\cE'_r(\cF)$ on a function $f\in C^\infty(L)$ is given by
	\[
	 \Op_L(a)f = \Op(a|_{H(\cF)_L})f = (a|_{H(\cF)_L}, s_{H(\cF)_L}^*f).
	\]
	In particular, if $f$ happens to be the restriction to $L$ of a function $\tilde{f}\in C^\infty(M)$, we get
	\begin{align*}
	 \Op_L(a)f &= (a|_{H(\cF)_L}, s_{H(\cF)_L}^*\iota_L^*\tilde{f}) \\
	  &= ({q_U}_*(\iota_L^*a), s_{H(\cF)_L}^*\iota_L^*\tilde{f}) \\
	  &= (\iota_L^*a, {q_U}^* s_{H(\cF)_L}^*\iota_L^*\tilde{f}) \\
	  &= (\iota_L^*a, \iota_L^* s_U^* \tilde{f}) \\
	  &= (\Op(a)\tilde{f})|_L,
	\end{align*}
	where $\Op(a)$ is the action of $a$ on $C^\infty(M)$ as defined in Proposition \ref{prop:Op}.  In other words, the action on the base $M$ and the action on the leaf $L$ are compatible in the sense that $\iota_L^*\circ \Op(a) = \Op_L(a)\circ \iota_L^*$.

	Similar constructions yield algebra morphisms $\cE'_s(\cF) \to \cE'_s(H(\cF)_L)$ and $\cE'_{r,s}(\cF) \to \cE'_{r,s}(H(\cF)_L)$.  Following the reasoning described in Definition \ref{def:Op2} and Proposition \ref{prop:generalized_functions}, we obtain that $\cE'_{r,s}(\cF)$ acts on the spaces of functions and generalized functions, $C^\infty(L)$, $\Cc^\infty(L)$, $C^{-\infty}(L)$, $\Cc^{-\infty}(L)$, as well as $C^\infty(\tilde{L})$, $\Cc^\infty(\tilde{L})$, $C^{-\infty}(\tilde{L})$ and $\Cc^{-\infty}(\tilde{L})$.

\bibliographystyle{alpha} 
\bibliography{paper}

\end{document}